\documentclass[12pt]{amsart}
\usepackage[hmargin=1in,vmargin=1in]{geometry}
\usepackage{url, mathrsfs, amsrefs}
\usepackage{hyperref}
\usepackage{fancyhdr} 
\usepackage[T1]{fontenc}
\usepackage[mathscr]{euscript} %this is to get the script style I prefer
\usepackage{amsmath, amssymb, amsthm}
\usepackage{mathtools}
	
\usepackage{graphicx,  epstopdf}
\usepackage{subfig}
\usepackage{color}
\usepackage{listings} %for blocks of code
	\definecolor{dkgreen}{rgb}{0,0.6,0}
	\definecolor{gray}{rgb}{0.5,0.5,0.5}
	\definecolor{mauve}{rgb}{0.58,0,0.82}
\usepackage{bbm}
\usepackage{subfloat}
\usepackage[draft]{fixme}
\usepackage{bm} 
\usepackage{epigraph}
\usepackage{ifpdf}
\usepackage{array}
\usepackage{xy}
\usepackage{amsbsy} 
\usepackage{multicol}
\usepackage{multirow}
\usepackage{supertabular}
\usepackage[makeroom]{cancel}

\newcommand{\M}{\mathscr{M}}
\newcommand{\V}{\mathcal{V}}
\newcommand{\PP}{\mathbb{P}}
\newcommand{\C}{\mathbb{C}}

\DeclareMathOperator{\Grass}{Gr}
\DeclareMathOperator{\GL}{GL}
\DeclareMathOperator{\Var}{\mathcal V}

\DeclareMathOperator{\GC}{GC}
\DeclareMathOperator{\sign}{sign}

\theoremstyle{plain}
\newtheorem{thm}{Theorem}[subsection] % change to \newtheorem{thm}{Theorem}[section] if you want thm 2.3 rather than thm 2.2.1 etc. 
\newtheorem*{thm*}{Theorem}
\newtheorem{prop}[thm]{Proposition}
\newtheorem{lem}[thm]{Lemma}
\newtheorem{prob}[thm]{Problem}

\theoremstyle{definition}

\newtheorem{ex}[thm]{Example}
\newtheorem{rmk}[thm]{Remark}

\title{Geometric Equations for Matroid Varieties}
\author{Jessica Sidman}
\address{Mount Holyoke College, South Hadley, MA 01075}
\email{jsidman@mtholyoke.edu}
\author{Will Traves}
\address{United States Naval Academy, Annapolis, MD 21402}
\email{traves@usna.edu}
\author{Ashley Wheeler}
\address{Mount Holyoke College, South Hadley, MA 01075}
\email{awheeler@mtholyoke.edu}
\begin{document}

\maketitle

% % % % %
\begin{abstract}
	Each point $x$ in $\Grass(r,n)$ corresponds to an $r \times n$ matrix $A_x$ which gives rise to a matroid $\M_x$ on its columns.  Gel'fand, Goresky, MacPherson, and Serganova showed that the sets $\{y \in \Grass(r,n) | \M_y = \M_x\}$ form a stratification of $\Grass(r,n)$ with many beautiful properties.  However, results of Mn\"ev and Sturmfels show that these strata can be quite complicated, and in particular may have arbitrary singularities.  We study the ideals $I_x$ of matroid varieties, the Zariski closures of these strata.  We construct several classes of examples based on theorems from projective geometry and describe how the Grassmann-Cayley algebra may be used to derive non-trivial elements of $I_x$ geometrically when the combinatorics of the matroid is sufficiently rich.
\end{abstract}
\section{Introduction}
\label{sec:intro}

 Let $x \in \Grass(r,n)$ be a point in the Grassmannian of $r$-planes in $\C^n.$   Define a matroid $\M_x$ on the columns of $A_x$, an $r \times n$ matrix whose rows are a basis for the subspace corresponding to $x$. Gel'fand, Goresky, Macpherson, and Serganova \cite{ggms}  introduced the matroid stratification of the Grassmanian by sets of the form
\[
\Gamma_x = \{ y \in \Grass(r,n) \mid \M_y = \M_x\},
\] 
and gave beautiful connections to combinatorics.  We will study the ideals of the Zariski closures of these strata, the matroid varieties $\V_x=\overline{\Gamma_x}$.

The ideal $I_x = I(\V_x)$ lies in the homogeneous coordinate ring of $\Grass(r,n)$, which is a quotient of a polynomial ring in which the variables are Pl\"ucker coordinates.  The Pl\"ucker coordinates of $x$ correspond to $r\times r$ minors of $A_x,$ and if $\lambda$ is an ordered subset of $\Omega_n=\{1,\dots,n\}$ of size $r$, we define  $[\lambda]$ to be the determinant of the submatrix of $A_x$ obtained by selecting columns indexed by $\lambda$.  
   
   If the columns indexed by $\lambda$ fail to be a basis of $\M_x$, then certainly $[\lambda] \in I_x,$ and we define the ideal 
\[
N_x = \langle [\lambda]  \mid \lambda \in\textstyle\binom{\Omega_n}{r}, \ \lambda \text{ is not a basis of $\M_x$} \rangle\subseteq I_x.
\] 
Note that the quotient ring $\C[\Grass(r,n)]/N_x$ is isomorphic to the ring  $B_{\M_x}$ defined by White in \cites{whiteBracketsI, whiteBracketsII}.  Knutson, Lam, and Speyer \cite{kls} show that the two ideals $N_x$ and $I_x$ are equal when $\M_x$ is a positroid.  However, Sturmfels \cite{sturmfels} produced an example of a representable matroid where $N_x$ is strictly smaller than $I_x$.  We say that an element of the coordinate ring of the Grassmannian is \emph{nontrivial} for the matroid $\M_x$ if it lies in the ideal $I_x$ but not in $N_x$. 

It seems beyond reach to hope to find explicit generators for $I_x$ for a general $x$.  Results of Mn\"ev \cite{mnev} and Sturmfels \cite{sturmfelsStrata} show that arbitrary singularities defined over the rationals may be found in matroid varieties.  Moreover, as Knutson, Lam, and Speyer \cite{kls} note, it is not even known if $\V_x$ is irreducible or equidimensional, and they refer to questions involving the strata as having ``paved the road to Hell'' leading into an ``abyss.''  Indeed, it can be devilishly difficult to compute the ideal of a matroid variety. 

Our results point to an interesting middle ground between positroids and matroids. We use results in classical incidence geometry, such as Pascal's Theorem, to produce nontrivial elements in the ideal $I_x$ of $\V_x$. The Grassmann-Cayley algebra, described in Section  \ref{sec: GC}, provides an algebraic language to state and prove results in incidence geometry. Section \ref{sec:pascal} is devoted to producing nontrivial polynomials in the matroid variety associated to Pascal's Theorem. Our core result, Theorem \ref{thm: pascalQuartic}, is both a prototype upon which later results are modeled and an ingredient in later proofs. In Section \ref{sec: generalizations} we state three generalizations of this result. In Theorem \ref{thm: morePoints} we define a matroid variety on an arbitrary number of points on a conic and apply Pascal's Theorem repeatedly to construct several nontrivial quartics in $I_x.$  In a different direction, in Theorem \ref{thm:pascalConfigGen} we use Caminata and Schaffler's recent results characterizing sets of points on a rational normal curve \cite{caminataSchaffler}  to provide examples in higher dimensions.   Finally, in Theorem \ref{thm: pascalCB}, we observe that Pascal's Theorem itself is a special case of a phenomenon arising from the Cayley-Bacharach Theorem \cite{eisenbud+green+harris} and use this result to produce an infinite set of examples in the plane where $N_x \subsetneq I_x$.

% % % % %
\section{Matroid varieties and the Grassmann-Cayley algebra}
\label{sec:2}
In \S \ref{sec: matroidVars} we introduce matroid varieties and their defining ideals. In \S\ref{sec: GC} we give a brief introduction to the Grassmann-Cayley algebra, which provides a language for translating synthetic geometric constructions into algebra.  For a more comprehensive introduction to the Grassmann-Cayley algebra see \cites{richter-Gebert,sturmfels,st}.
% % %
\subsection{Matroid varieties}
\label{sec: matroidVars}

Recall that a \emph{matroid} $\M$ may be specified by giving a finite ground set $\Omega_n$ and a nonempty collection $\mathscr B$ of subsets of $\Omega_n$, satisfying the \emph{exchange axiom}: 	If $B,B'\in\mathscr B$ and $\beta\in B\setminus B'$, then there exists $\beta'\in B'\setminus B$ such that $(B\setminus \{\beta\})\cup\{\beta'\}\in\mathscr B$.
The sets in $\mathscr B$ are called \emph{bases}, and it follows from the definitions that all of the elements of $\mathscr B$ have the same cardinality, which we refer to as the \emph{rank} of $\M.$

\begin{ex}
	\label{ex:matrixMatroid}
Fix $r\leq n$ and let $A$ denote an $r\times n$ matrix.  Then the indices of the columns of $A$ form the 
ground set $\Omega_n$ of a matroid $\M$.  The bases of $\M$ are sets of indices of columns that form a basis 
for the column space of $A$.  If the columns of $A$ 
correspond to $n$ points in $\mathbb{C}^r$ in general position, this construction produces the 
\emph{uniform matroid of rank $r$} whose bases are the $r$-element subsets of $\Omega_n.$
\end{ex}

If $x \in \Grass(r,n)$ let $A_x$ be any $r\times n$ matrix whose rows are a basis for the subspace corresponding to $x.$  The matrix $A_x$ is only well-defined up to left multiplication by an element of $\GL_r,$ which performs an invertible linear combination of the rows.  This changes the column space of $A_x$, but preserves the subsets of columns of $A_x$ that are linearly independent,  so the matroid on the columns of $A_x$ is invariant. We let $\M_x$ denote the matroid on the columns of $A_x$ as in Example \ref{ex:matrixMatroid}.  The set $\Gamma_x$ contains all points $y \in \Grass(r,n)$ for which $\M_x = \M_y,$ and the \emph{matroid variety} $\V_x$ is the closure $\V_x = \overline{\Gamma_x} \subset \Grass(r,n)$ in the Zariski topology.

\begin{ex}[Sturmfels \cite{sturmfelsStrata}, Ford \cite{ford}]
	\label{ex:pencil}
	Suppose $P_1,\dots, P_7$ are points in the projective plane $\PP^2$ such that the lines $\overline{P_1P_2},\overline{P_3P_4},\overline{P_5P_6}$ meet in the point $P_7$ and no three of the points $P_1, \ldots, P_6$ are collinear, as depicted in Figure \ref{fig:pencil}. 
		\begin{figure}[h!t]
		\includegraphics[scale=0.225]{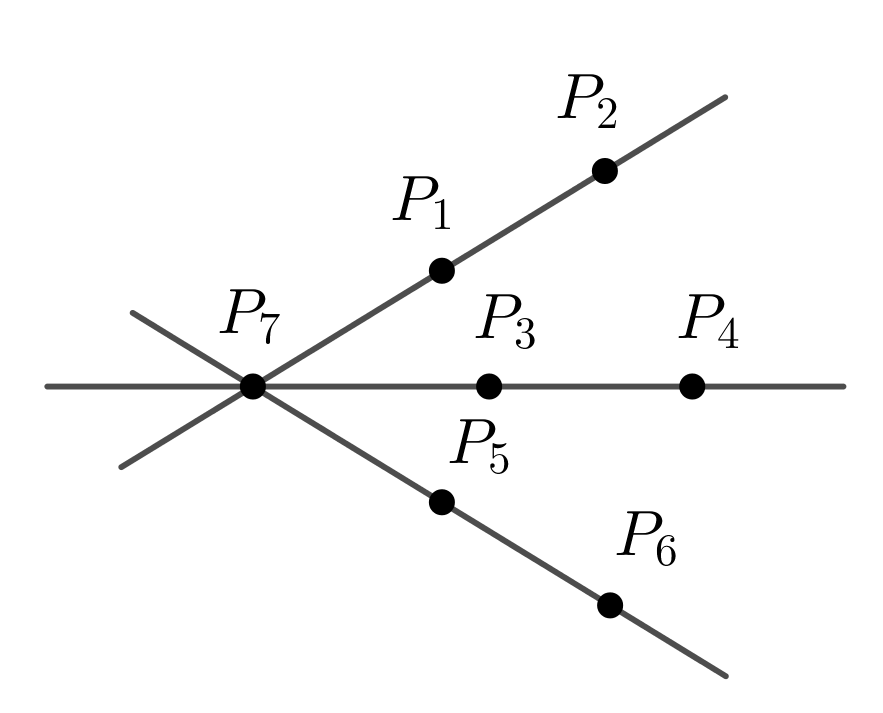}
		\caption{A pencil of lines meeting in a marked point.}
		\label{fig:pencil}
		\end{figure}
	Fix homogeneous coordinates for the seven points and write them as the columns of a $3\times 7$ matrix $A_x$.  Let $x\in\Grass(3,7)$ be the row span of $A_x$ and let $\M_x$ be the matroid associated to $x$.  The bases for $\M_x$ consist of all subsets of $\Omega_7$ of size 3, except for $\{1,2,7\}$, $\{3,4,7\}$, and $\{5,6,7\}$.
\end{ex}

We introduce notation to describe the ideal of a matroid variety in terms of brackets.  Let $m = \binom{n}{r}-1.$  If $x \in \Grass(r,n)$ then the $r \times r$ minors of $A_x$ are the Pl\"ucker coordinates of $x$ in the embedding $\Grass(r,n) \hookrightarrow \PP^m.$  Let $\Lambda = \Lambda(\Omega_n,r)$ be the collection of all $r$-element subsets $\lambda$ of $\Omega_n$.  We use the bracket $[\lambda_1 \cdots \lambda_r]$ to denote the coordinate corresponding to the $r \times r$ minor of $A_x$ whose columns are indexed by $\lambda_1 < \cdots <\lambda_r.$  We will abuse notation and sometimes write the elements in a bracket out of order, which is equivalent to a bracket with the usual ordering up to a sign given by the appropriate permutation of the elements of $\lambda.$  e.g., $[132] =-[123].$  We denote the homogeneous coordinate ring of $\PP^m$ in bracket coordinates by $\C[\Lambda]$ and let
$I_{r,n}$ denote the ideal of Pl\"ucker relations for $\Grass(r,n)$ so that $\C[\Grass(r,n)] = \C[\Lambda]/I_{r,n}$. Sturmfels \cite{sturmfels}*{Theorem 3.1.7} gives an explicit Gr\"obner basis for $I_{r,n}$. In particular, for any distinct $\lambda_1,\dots,\lambda_5\in\Omega_n$,
\begin{equation} 
\label{GPrel} 
[\lambda_1 \lambda_2 \lambda_3][\lambda_1 \lambda_4 \lambda_5] - [\lambda_1 \lambda_2 \lambda_4][\lambda_1 \lambda_3 \lambda_5] + [\lambda_1 \lambda_2 \lambda_5][\lambda_1 \lambda_3 \lambda_4] \in I_{3,n}. 
\end{equation} 
See Richter-Gebert \cite{richter-Gebert}*{Theorem 6.3} for a derivation of this \emph{Grassmann-Pl\"ucker relation} using Cramer's rule.   

The ideal $N_x = \langle [\lambda] \mid \lambda \textrm{ is not a basis for } \M_x\rangle \subset \C[\Grass(r,n)]$ is the ideal generated by brackets corresponding to non-bases of $\M_x.$ Let $J_x% = \langle [\lambda] \mid \lambda \textrm{ is a basis for } \M_x \rangle 
\subset \C[\Grass(r,n)]$ be the ideal generated 
by the product of the brackets corresponding to bases of $\M_x.$ We let $I_x\subset\C[\Grass(r,n)]$ denote the ideal of $\V_x= \overline{\Gamma_x}.$

We can use $N_x$ and $J_x$ to get information about $I_x$. Recall that the \emph{saturation} of an ideal $I$ contained in a ring $R$ by an ideal $J\subset R$ is $$I:J^{\infty} \stackrel{\text{def}}{=}  \langle f \in R: \text{for each } g \in J, \; g^kf \in I \text{ for some } k\in \mathbb{Z}_+ \rangle.$$ 

\begin{prop} We have the ideal containments $(N_x : J_x^{\infty}) \subseteq (I_x:J_x^\infty) = I_x$, and $\sqrt{(N_x : J_x^{\infty}) } = I_x.$
\end{prop}

\begin{proof} The first containment follows from $N_x \subseteq I_x$ and the containment $I_x \subseteq (I_x:J_x^\infty)$ is trivial so we concentrate on showing $(I_x:J_x^\infty) \subseteq I_x$. 
Let $g$ denote the generator of $J_x$. To show the second containment it is enough to show that if $g^kf \in I_x$ then $f \in I_x$. Assume $g^kf \in I_x$. Then for all $y \in {\Gamma}_x$, $(g^kf)(y) = \left( g(y) \right)^k f(y) = 0$, but $g$ does not vanish on ${\Gamma}_x$ so $f(y) = 0$. Since $f$ vanishes on ${\Gamma}_x$, $f \in I({\Gamma}_x) = I_x$.

Now we show that $\V(N_x: J_x^{\infty}) = \V(I_x)$. The variety $\V(N_x:J_x^\infty)$ is the Zariski-closure of $\V(N_x) \setminus \V(J_x).$  The set difference $\V(N_x) \setminus \V(J_x)$ is the set of points $y$ on which the brackets $[\lambda]$ corresponding to non-bases $\lambda$ in $\M_x$ vanish but none of the brackets $[\lambda]$ corresponding to bases in $\M_x$ vanish, i.e., $\V(N_x) \setminus \V(J_x) = \Gamma_x.$ Therefore, $\V(N_x: J_x^\infty) = \overline{\Gamma_x} = \V_x = \V(I_x)$. Now by Hilbert's Nullstellensatz, the radical $$\sqrt{N_x: J_x^\infty} = \{ f \in \C[\Grass(r,n)]: 
f^t \in (N_x:J_x^\infty) \text{ for some } t \in \mathbb{Z}_+ \}$$ equals $I_x$. 
 \end{proof}

We highlight two hard problems concerning saturating by $J_x$.

\begin{prob}
When is $I_x=(N_x:J_x^{\infty})$? That is, when is the saturation $(N_x:J_x^\infty)$ radical?
\end{prob}

Part of the difficulty in this problem is that it is very hard to explicitly compute the Zariski closure of sets in the Grassmannian. The problem has a straightforward answer when $N_x$ is radical. In that case, $(N_x:J_x^\infty)$ is also radical: if $g$ is the generator of $J_x$ and $f^n \in (N_x:J_x^\infty)$ then $f^n$ multiplies some power $g^k$ into $N_x$ but then $(fg)^{k+n} \in N_x$, so $fg \in N_x$ (as $N_x$ is radical), and thus $f \in (N_x:J_x^\infty)$.
%Indeed, we saw an example of an unusual point $y$ in $\overline{\Gamma_x}$ in Example \ref{ex: pencilIdeal}.  
White conjectured (Conjecture 6.8B in \cite{whiteBracketsII}) that $N_x$ is radical if $\M_x$ is a unimodular combinatorial geometry; that is, $N_x$ is radical if $\M_x$ can be represented by a matrix $A_x$ so that no column of $A_x$ is a scalar multiple of another and all the minors of $A_x$ are -1, 0, or 1.  However, it is a difficult problem to completely characterize when the ideal $N_x$ is radical. 
% Combinatorial geometry means that we have a matroid in which every subset of size at most 2 is independent, so no vector is a scalar multiple of another (and in particular, no zero vectors are allowed). Unimodular means that the matrix is totally unimodular, i.e. all its minors are 0, 1 or -1. 

\begin{prob}
Bound the degree of the generators of the saturation $(N_x:J_x^{\infty})$ in terms of the combinatorics of the matroid $\M_x$. 
\end{prob}
 
The saturation of $N_x$ by $J_x$ may produce nontrivial elements of $I_x$.
\begin{ex}
	\label{ex: pencilIdeal} 
Let $\V_x\subset \Grass(3,7)$ denote the variety associated to the matroid $\M_x$ in Example \ref{ex:pencil}.  
Since the sets $\{1,2,7\}, \{3,4,7\}$, and $\{5,6,7\}$ are not bases, the brackets $[127], [347],$ and $[567]$ 
generate the ideal $N_x.$ 
Note that if $[\lambda_1 \lambda_2 \lambda_3] \in N_x$ then Equation (\ref{GPrel}) implies 
that $ [\lambda_1 \lambda_2 \lambda_5][\lambda_1 \lambda_3 \lambda_4] - [\lambda_1 \lambda_2 \lambda_4]\
[\lambda_1 \lambda_3 \lambda_5] \in N_x$. Now $$ \begin{array}{llll}  {[127]} \in N_x & \Rightarrow & a-b 
& \stackrel{\text{def}}{=} {[123][247] - [124][237]} \in N_x \\
{[347]} \in N_x & \Rightarrow & c-d & \stackrel{\text{def}}{=} {[237][467] - [367][247]} \in N_x \\
{[567]} \in N_x & \Rightarrow & e-f & \stackrel{\text{def}}{=} {[456][367] - [356][467]} \in N_x, 
\end{array}$$ so $$ [237][367][247][467]\left([123][456]-[124][356]\right) = (a-b)ce + (c-d)be + (e-f)bd$$ 
lies in $N_x.$ Now the quadratic $F = [123][456]-[124][356]$ lies in $(N_x: J^\infty) \subseteq (I_x: J^\infty) = 
I_x$. In the next section we use the Grassmann-Cayley algebra to give another proof that $F$ is in $I_x$.

We now show that $F$ is indeed nontrivial. To see that $F \notin N_x$, choose six points in linearly general position in $\PP^2$, with the property that they cannot be partitioned into three pairs spanning coincident lines. 
Define $y \in \Grass(3,7)$ by constructing a matrix $A_y$ in which the first six columns are homogeneous coordinates of these points and the last column is zero.  Because all brackets generating $N_x$ involve the seventh column of $A_y$, we have that $y$ is in the algebraic set $\Var(N_x)$.  However, by construction, $y$ does not satisfy the polynomial $F$.  Now the following remark shows that $F\notin N_x$.  

\begin{rmk}
	\label{rmk: fnotinNx}
	If $y \in \Var(N_x)$ with $F(y) \neq 0$, then $F \not\in I(\Var(N_{x})) = \sqrt{N_{x}} \supset N_{x}$, so $F \not\in N_x$. This argument will be used throughout the paper to show that an element of $I_x$ is nontrivial.       
\end{rmk}

In fact, Ford \cite{ford} claimed that $I_x = N_x+\langle F \rangle$. We give a computational argument to verify Ford's description of the ideal $I_x$.  

\begin{thm}\label{thm:Ford} Using the notation of Example \ref{ex: pencilIdeal}, we have $I_x = N_x + \langle F \rangle.$
\end{thm}

\begin{proof}
In what follows, all of the ideals have generators with rational coefficients.  Therefore, the computations performed in Macaulay2 \cite{M2} using Gr\"obner methods produce the same results over $\mathbb{Q}$ or $\C.$    We found that the saturation of $\langle [127],[347],[567] \rangle + I_{3,7}$ in $\mathbb{Q}[\Lambda]$ by the product of the brackets corresponding to bases is precisely $\langle [127],[347],[567], F \rangle + I_{3,7}$, so the saturation $(N_x: J_x^\infty)$ in the quotient ring $\C[\Grass(3,7)]$ is $N_x + \langle F \rangle$. We also computed the initial ideal of $\langle [127],[347],[567] \rangle + I_{3,7}$ in $\mathbb{Q}[\Lambda]$. That initial ideal is a squarefree monomial ideal so both the initial ideal and the original ideal are radical ideals. Taking the quotient by $I_{3,7}$, we see that $N_x$ is a radical ideal in $\mathbb{C}[\Grass(3,7)]$. Then $(N_x:J_x^\infty)$ is also radical and 
$I_x = \sqrt{(N_x:J_x^\infty)} = (N_x:J_x^\infty) = N_x + \langle F \rangle$. 
\end{proof}

\end{ex}

\begin{rmk}
	In fact, localizing the coordinate ring of $\V_x$ at brackets corresponding to bases we get a ring isomorphic to a polynomial ring in nine variables localized at three variables, showing that $\V_x$ is irreducible and $I_x$ is prime. 
\end{rmk}

% % % 
\subsection{The Grassmann-Cayley algebra}
\label{sec: GC}

The \emph{Grassmann-Cayley algebra} $\GC(\C^r)$ is the usual exterior algebra over the vector space $\C^r$ equipped with two operations, the join and the meet. Working with $\GC(\C^r)$  allows us to construct polynomials in the coordinate ring of $\Grass(r,n)$ using geometry.

The \emph{join} operation is the exterior product, though the symbol $\vee$ is used for the join in the Grassmann-Cayley algebra rather than the symbol $\wedge$.  The join of $k$ vectors in $\C^r$ is an \emph{extensor} of \emph{step} $k$ and is non-zero if and only if the vectors are linearly independent.  We will often suppress notation and denote the join via juxtaposition.  If $\{e_1, \ldots, e_r\}$ is the standard basis for $\C^r$, then $v_1 \vee \cdots \vee v_r = [v_1 \cdots v_r]e_1 \cdots e_r$, and we identify an extensor of step $r$ with a bracket.

 Given two extensors $v=v_1\cdots v_k$, $w=w_1\cdots w_\ell$, when $k+\ell < r$ the \emph{meet} $\wedge$  is defined to be 0; if $k+\ell \geq r$, the meet is defined to be
\[
v\wedge w = \sum_{\sigma\in\mathscr S(k,\ell,r)}\sign(\sigma)[v_{\sigma(1)}\cdots v_{\sigma(r-\ell)}w_1\cdots w_\ell]\cdot v_{\sigma(r-\ell+1)}\cdots v_{\sigma(k)}, 
\]
where $\mathscr S(k,\ell,r)$ is the set of all permutations $\sigma$ of $\Omega_k$ so that $\sigma(1) < \cdots < \sigma(r-\ell)$ and $\sigma(r-\ell+1) < \cdots < \sigma(k)$. The meet of two extensors is again an extensor (though this is not obvious), and the vectors in this extensor form a basis for the intersection of the two subspaces whose bases are given by the two extensors in the meet. 

While Grassmann-Cayley expressions whose step is a multiple of $r$ can be written in terms of brackets, a polynomial in the bracket algebra may or may not have a \emph{Cayley factorization} into a Grassmann-Cayley expression involving only meets and joins.   

\begin{ex}
\label{ex: pencilEq} We use geometry to recover the nontrivial polynomial $F$ from Example \ref{ex:pencil}.
In that example, the lines $\overline{P_1P_2}$, $\overline{P_3P_4}$, and $\overline{P_5P_6}$ are coincident.   Consider the Grassmann-Cayley expression $(34 \wedge 12) \vee 56.$  The expression $(34 \wedge 12)$ is an extensor of step 1 representing the point of intersection, $P_7,$ of the lines $\overline{P_3P_4}$, and $\overline{P_1P_2}$.  The join of this point with the line $\overline{P_5P_6}$ is an extensor of step 3 given by the bracket $[567]$ which is zero precisely when $P_5, P_6,$ and $P_7$ fail to span $\PP^2.$   

Therefore, the vanishing of the expression encodes the condition that the three lines are coincident, and
\[
(34 \wedge 12) \vee 56 = \left([312]4 - [412]3 \right) \vee 56 = [123][456] -[124][356] = F,
\] 
which is the polynomial given by Ford. Since $F$ vanishes on $\Gamma_x$, $F \in I_x$. 
\end{ex}

% % % % % 
\section{A matroid variety from Pascal's Theorem} 
\label{sec:pascal}

Blaise Pascal rose to prominence by proving an incidence theorem involving points on a conic. Braikenridge and Maclaurin independently proved the converse. Our goal in this section is to produce nontrivial polynomials in the ideal associated to a matroid variety coming from their result.

\begin{thm}[Pascal, Braikenridge-Maclaurin]
\label{thm: pascal}
Six points, $P_1, \ldots, P_6,$ lie on a conic if and only if the points  $P_7 = \overline{P_1P_2} \cap \overline{P_4P_5}$,  $P_8 = \overline{P_2P_3} \cap \overline{P_5P_6}$,  and $P_9 = \overline{P_3P_4} \cap \overline{P_6P_1}$ are collinear, as depicted in Figure \ref{fig:Pascal}. Equivalently, the six points lie on a conic if and only if  
$(12 \wedge 45) \vee (23 \wedge 56) \vee (34 \wedge 61) =0$.  In the coordinate ring of the Grassmannian, the condition for six points to lie on a conic reduces to the vanishing of the binomial 
\begin{equation} \label{eqn:quartic}
f = [123][145][246][356]-[124][135][236][456]. 
\end{equation}
\end{thm}
 
\begin{figure}[h!t]
\begin{center}
\includegraphics[scale=0.24]{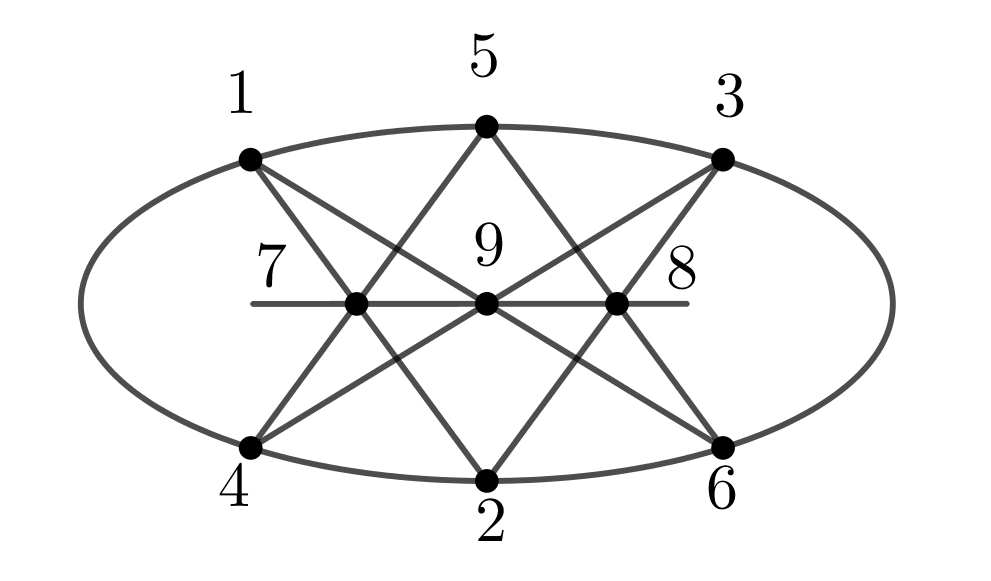}
\caption{The six points $P_1, \ldots, P_6$ lie on a conic precisely when points $P_7$, $P_8$ and $P_9$ are collinear.}
\label{fig:Pascal}
\end{center}

\end{figure} 
 
Let $x_{\rm Pascal}\in\Grass(3,9)$ be the row span of the $3\times 9$ matrix whose columns are homogeneous coordinates for points $P_1, \ldots, P_9$ arranged as in Figure \ref{fig:Pascal} so that points $P_1, \ldots, P_6$ lie on a conic and there are just 7 collinear triples among the nine points. Define $\M_{\rm Pascal} = \M_{x_{\rm Pascal}}$ to be the matroid on the nine points. 

\begin{thm} 
\label{thm: pascalQuartic} 
The defining ideal $I_{\rm Pascal}$ for the matroid variety given by $\M_{\rm Pascal}$ contains at least one quartic, three independent cubics, and three independent quadrics not in the ideal $N_{\rm Pascal}.$   
\end{thm}

\begin{proof}
From Figure \ref{fig:Pascal} we see that 
\[
	N_{\rm Pascal} = \langle [127], [238], [349],  [457], [568], [169], [789]\rangle.
\]
\noindent {\bf The quartic}: Let $f$ be the expression in Equation \eqref{eqn:quartic}. 
As explained in Remark \ref{rmk: fnotinNx}, in order to show that $f\notin N_{\rm Pascal}$, it suffices to find $y \in \Var(N_{\rm Pascal})$ with $f(y) \neq 0$. Let $y$ be the point in $\Grass(3,9)$ associated to the $3 \times 9$ matrix $A_y$ whose first six columns correspond to six points in $\PP^2$ that do not lie on a conic and whose last 3 columns are zero.   
Each of the seven brackets in $N_{\rm Pascal}$ vanishes on $y$ because each bracket involves at least one of the points $7$, $8$, or $9$.  
However, Pascal's Theorem does not hold for the first six points, so $f(y) \neq 0$, as required.  
\smallskip

\noindent {\bf Three cubics}: 
Replacing the respective parenthesized expressions in the Grassmann-Cayley expression for $f$ by $7$, $8$, or $9$ produces a cubic.  For example, we have 
$$
\begin{array}{lll}
g_7 & = & 7 \vee (23 \wedge 56) \vee (34 \wedge 61) \\
&  =  & [256][361][734]-[356][361][724]+[356][[461][723] 
\end{array} 
$$
in $I_{\rm Pascal}$.
We show the cubic $g_7$ is nontrivial, 
and by similar constructions it follows that the other two cubics are also nontrivial.  As before, it is enough to find $z\in\Var(N_{\rm Pascal})$ with $g_7(z)\neq 0$. Consider the point $z\in\Grass(3,9)$ with representative matrix $A_z$ given as follows: columns $2,\dots,6$ are coordinates for points in general position, columns $1$ and $7$ are equal and on the line joining $4$ and $5$, and columns $8$ and $9$ equal zero.  The nonzero points are depicted in the left side of Figure \ref{fig: cubic}.  
\begin{figure}[h!t]
  \includegraphics[scale=0.15]{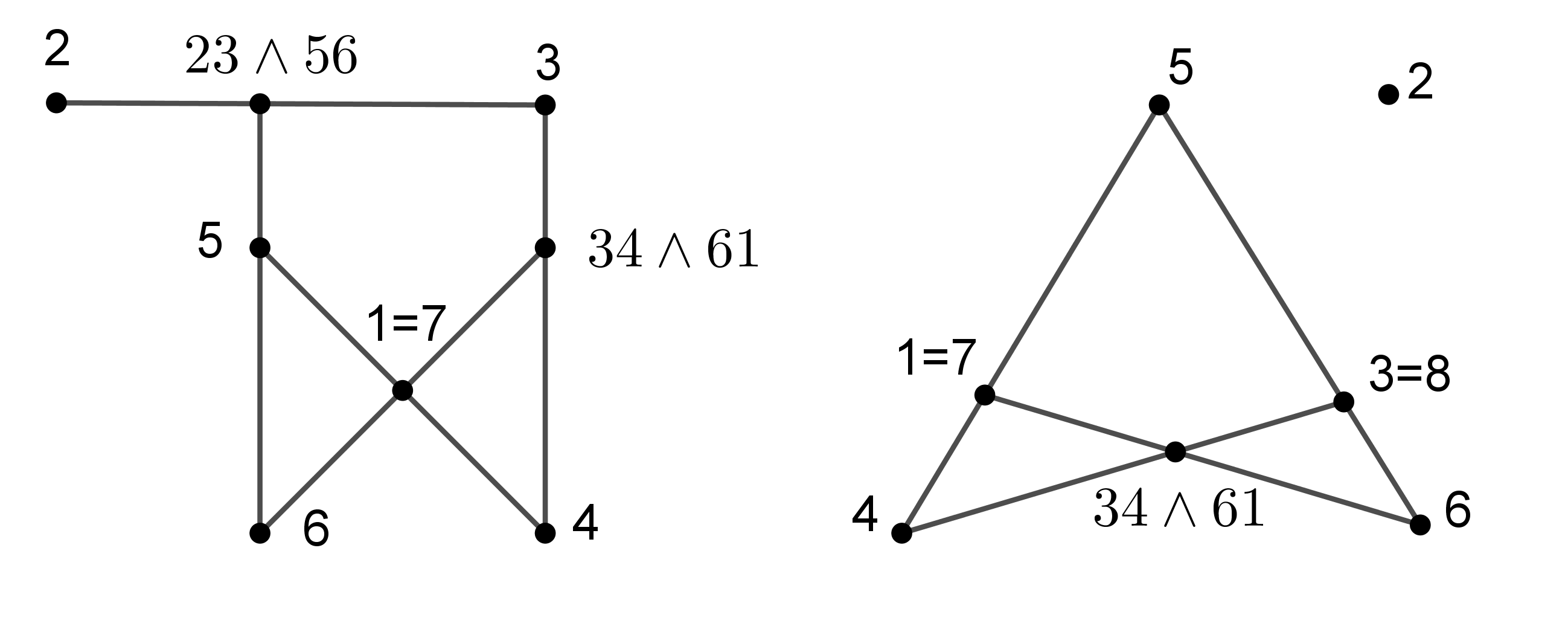}
\caption{Two configurations used to show cubics (left) and quadratics (right) are nontrivial.}
\label{fig: cubic}
\end{figure}

The brackets in $N_{\rm Pascal}$ vanish on $z$ so $z \in \Var(N_{\rm Pascal})$. The Cayley factorization of $g_7$ shows that $g_7$ vanishes precisely when $7$ is collinear with $23 \wedge 56$ and $34 \wedge 61,$ which we can see fails in the left side of Figure \ref{fig: cubic}.  It follows that $g_7(z) \neq 0$ and $g_7 \in I_{\rm Pascal} \setminus N_{\rm Pascal}$ is a nontrivial cubic. 

We also claim that the three cubics are independent.  Let 
$$
g_8 = (12\wedge 45) \vee 8\vee (34 \wedge 61) \text{ and } g_9 = (12\wedge 45) \vee (23 \wedge 56) \vee 9. 
$$
If $g_7$ were a $\C[\Grass(3,9)]$-combination of $g_8$ and $g_9$, then we would have polynomials $\varphi,\psi\in\C[\Grass(3,9)]$ such that 
$
g_7 = \varphi g_8 + \psi g_9
$ 
at all points in $\Grass(3,9)$. Since points $8$ and $9$ are zero, $g_8(z) = g_9(z) = 0$ and $(\varphi g_8 + \psi g_9)(z) = 0$. However, $g_7(z) \neq 0$, contradicting our dependence assumption. Similar arguments prove that $g_8$ and $g_9$ are independent of each other.
\smallskip

\noindent {\bf Three quadrics}: To construct a nontrivial quadric in $I_{\rm Pascal},$ replace two meets in $f$ by points to produce
\begin{equation}
\label{eqn: conic}
h_{78} = 7 \vee 8\vee (34 \wedge 61) = [748][361]-[461][738].
\end{equation}
We see that we obtain three quadrics this way, as we can replace any two of the three meets by points.  We  show that $h_{78}$ is nontrivial by finding $w\in\Grass(3,9)$ in $\Var(N_{\rm Pascal})$ with $h_{78}(w) \neq 0$. Let $w \in \Grass(3,9)$ denote the point with representative matrix consisting of columns $2$, $4$, $5$, and $6$ in general position, columns $1$ and $7$ equal and on the line joining $4$ and $5$, columns $3$ and $8$ equal and on the line joining $5$ and $6$, and $9$ equal to the zero vector, as depicted in the right side of Figure \ref{fig: cubic}. Using the figure it is easy to see that the points (together with the zero vector 9) satisfy the brackets generating $N_{\rm Pascal}$, so $w \in \Var(N_{\rm Pascal})$.  Since the bracket polynomial $h_{78}$ vanishes exactly when $7, 8$ and $34 \wedge 61$ are collinear we see that $h_{78}(w)\neq 0.$

Now let $h_{79} = 7 \vee (23 \wedge 56) \vee 9$ and $h_{89} = (12 \wedge 45) \vee 8 \vee 9.$  Since point $9$ is the zero vector, we can see from the Cayley factorization that $h_{79}(w) = h_{89}(w) = 0,$ which implies that $h_{78}$ is independent of the other two quartics.  By symmetry, we see that the three polynomials are independent.
\end{proof}

\begin{rmk} Using Macaulay2 \cite{M2} we checked that none of the cubics are in the ideal generated by the quadratics and $N_{\rm Pascal}$. We also checked that the quartic is not in the ideal generated by the quadratics, the cubics and $N_{\rm Pascal}$. This shows that the seven polynomials are not just degree-wise independent but they are also independent over the coordinate ring of the Grassmannian. 

\end{rmk}

% % % % % 
\section{Generalizations}
\label{sec: generalizations}

% % %
\subsection{More Points}
\label{sec: morePoints}

We generalize Theorem \ref{thm: pascalQuartic} to matroids of $n>6$ points on a conic, deriving many quartic, cubic and quadric equations that vanish on the associated matroid variety.  

\begin{thm}
\label{thm: morePoints}
	Let $\mathcal P = \{P_1, \ldots, P_n\}$ be distinct points lying on a nondegenerate conic. 
	For each $i \in \{6,\ldots, n\}$  
	define the sets 
\[
\mathcal Q_i = \left\{Q_{i1} = \overline{P_1P_2}\cap\overline{P_4P_5}, \ Q_{i2} = \overline{P_2P_3}\cap\overline{P_5P_i}, \ Q_{i3} = \overline{P_3P_4}\cap\overline{P_iP_1}\right\},
\]
and let $\mathcal Q = 
\bigcup_{i = 6}^n \mathcal Q_i$, and $q = |\mathcal Q|.$  
Define $x_n$ to be the point of $\Grass(3,n+q)$ associated to a fixed $3\times (n+q)$ matrix whose columns are the homogeneous coordinates for the points in $\mathcal P\cup \mathcal Q.$  Then the ideal $I_{x_n}$ contains independent quartic polynomials that are not in $N_{x_n}.$
\end{thm}

\begin{proof}
Applying Theorem \ref{thm: pascal} to each set $\{P_1, \ldots, P_5, P_i\} \cup \mathcal Q_i$ gives a quartic bracket polynomial $f_i \in I_{x_n}$ only involving $P_1, \ldots, P_5$ and $P_i$, guaranteeing that these six points are on a conic.  Since the first five points determine a unique nondegenerate conic, all of the points $P_1, \ldots, P_n$ must be on the same conic.  If any three points in $\mathcal P$ are collinear then the line must be a component of the conic, contradicting our nondegeneracy assumption. So no three points in $\mathcal P$ are collinear and 
every dependent triple of elements in $\mathcal P \cup \mathcal Q$ must contain an element of $\mathcal Q.$ 

To see that none of the $f_i$ are in $N_{x_n}$ define a $3 \times (n+q)$ matrix $A_{y_i}$ whose first $n$ columns are homogeneous coordinates for elements of $\PP^2$ with the property that $\mathcal P\setminus P_i$ lie on a conic not containing $P_i$, %no six lie on a conic
and whose last $q$ columns are zero. Note that each of the nonbases in $\M_{x_n}$ involves an element of $\mathcal Q$, so $y_i \in \Var(N_{x_n})$. However, $f_i(y_i) \neq 0$ by construction. It follows from Remark \ref{rmk: fnotinNx} that each $f_i$ is in $I_{x_n} \setminus N_{x_n}$. Futhermore, the polynomials $f_i$ are independent, since $f_j(y_i) \neq 0$ precisely when $i \neq j$.  
\end{proof}

Using the same construction as in Theorem \ref{thm: pascalQuartic}, for each $f_i$ we can construct three independent cubics and three independent quadrics in $I_{x_n}\setminus N_{x_n}$.

% % %
\subsection{Higher Dimension}

Caminata and Schaffler \cite{caminataSchaffler} recently proved a result that generalizes both Pascal's Theorem and the Braikenridge-Maclaurin Theorem to higher dimensions.  We use their result on rational normal curves to construct an infinite family of matroid varieties with nontrivial polynomials that can be constructed via the Grassmann-Cayley algebra, generalizing Theorem \ref{thm: pascalQuartic} to all dimensions.

We recall notation for a rational normal curve of degree $d$ in $\PP^d.$ 
Let $F_0, \ldots, F_d$ be a basis for the homogeneous forms of degree $d$ on $\PP^1.$  The image of the map $\nu_d: \PP^1 \to \PP^d$ given by $\nu_d([x_0:x_1]) = [F_0: \cdots: F_d]$ is a \emph{rational normal curve} of degree $d.$  Every subset of the rational normal curve $\nu_r(\PP^1) \subset \PP^r$  consisting of $r+1$ or fewer points is linearly independent.  As a consequence, the matroid associated to any finite subset of $r+1$ or more points in $\nu_r(\PP^1)$ is the uniform matroid of rank $r+1$ (see Example \ref{ex:matrixMatroid}). It is possible to construct a family of rational normal curves whose limiting position decomposes as a union of rational normal curves in proper subspaces of our ambient projective space. A \emph{quasi-Veronese variety} is a union of curves that are rational normal curves in their spans.

\begin{thm}[Caminata-Schaffler \cite{caminataSchaffler}*{Theorem 5.1}]
\label{thm: caminataSchaffler}
Suppose that $d \geq 2$ and $\mathcal P = \{P_1, \ldots, P_{d+4}\} \subset\PP^d$ do not lie on a hyperplane.  For each $\lambda\in\Lambda(\Omega_{d+4},6)$ and complementary set $\Omega_{d+4}\setminus\lambda = \{j_1 < \cdots < j_{d-2}\}$, define $H_{\lambda} = j_1 \cdots j_{d-2}.$   Then $\mathcal P$ lies on a quasi-Veronese curve if and only if for each $\lambda = \{i_1 < \cdots < i_6 \} \in\Lambda(\Omega_{d+4},6)$
\begin{equation}
\label{eq: caminataSchaffler}
(i_1i_2\wedge i_4i_5H_{\lambda} )\vee   (i_2i_3 \wedge i_5i_6H_{\lambda}) \vee (i_3i_4 \wedge i_6i_1H_{\lambda}) \vee H_{\lambda}
=0.
\end{equation} 
In particular, these Grassmann-Cayley conditions are satisfied for sets of points lying on a rational normal curve.  
\end{thm}

Let $P = \{P_1, \ldots, P_{d+4}\} \subset\PP^d$ lie on a rational normal curve.  Using the notation of Theorem \ref{thm: caminataSchaffler}, consider the set 
\begin{equation}
\label{eq: Q}
\mathcal{Q} = \bigcup_{\stackrel{\lambda=\{i_1,\dots,i_6\}}{\in\Lambda(\Omega_{d+4},6)}} 
\left\{ i_1i_2\wedge i_4i_5H_{\lambda}, \   i_2i_3\wedge i_5i_6H_{\lambda}, \  i_3i_4\wedge i_6i_1H_{\lambda}\right\}.
\end{equation}
Each element of $\mathcal Q$ is a point obtained by intersecting a line and a hyperplane.  Some points are repeated; for example, when $r=3$, choosing the 6-element sets $\{1,2,3,4,5,7\}$ and $\{1,2,3,4,6,7\}$ and using $i_1i_2 \wedge i_4i_5H_\lambda$ yields $12 \wedge 456$ and $12 \wedge 465,$ respectively, and these two points are the same.

\begin{thm}
\label{thm:pascalConfigGen}
Let $P = \{P_1, \ldots, P_{d+4}\} \subset\PP^d$ lie on a rational normal curve. Let $\mathcal Q$ denote the points given in Equation \eqref{eq: Q} and let $n = d+4+|\mathcal Q|$.   Define $A_{\rm CS}$ to be the matrix whose columns are the points in $\mathcal P$ followed by the points in $\mathcal Q$, where $x_{\rm CS} \in \Grass(d+1,n)$ corresponds to the subspace spanned by the rows of $A_{\rm CS}.$ Then there are nontrivial quartics in $I_{\rm CS}$.  
\end{thm}
\begin{proof}
Since the points in $\mathcal P$ lie on a rational normal curve, any subset of $d+1$ of them are linearly independent.  Hence, if a subset of $\mathcal P \cup \mathcal Q$ of size $d+1$ fails to be a basis, it must involve at least one point of $\mathcal Q$.  Hence, each of the brackets in $N_{\rm CS}$ must involve a point in $\mathcal Q.$

To show that $N_{\rm CS} \neq I_{\rm CS},$ choose $d+4$ points in $\PP^d$ that do not lie on a quasi-Veronese  curve, and define a $(d+1) \times n$ matrix $A_y$ whose first $d+4$ columns are homogeneous coordinates for the points in $\PP^d$ and whose remaining columns are zero.  Let $y$ be the point in $\Grass(d+1,n)$ corresponding to this matrix.  

Since each bracket in $N_{\rm CS}$ contains an index from $\mathcal Q$, and all of these columns of $A_y$ are zero, $y \in \Var(N_{\rm CS}).$  However, since the first $d+4$ columns of $A_y$ do not lie on a rational normal curve, the quartics in Equation \eqref{eq: caminataSchaffler} of Theorem \ref{thm: caminataSchaffler} are not satisfied. By Remark \ref{rmk: fnotinNx}, these quartics are in $I_{\rm CS}$ but not $N_{\rm CS}$. 
\end{proof}

\begin{ex}[Twisted cubic]
Theorem \ref{thm: caminataSchaffler} reduces to the Braikenridge-Maclaurin Theorem in the $d=2$ case.  Here, we work through the $d=3$ (twisted cubic) case.  The set $\mathcal P$ consists of seven points, and Theorem \ref{thm: caminataSchaffler} yields seven equations
\begin{equation}
 \label{eqn: 7eqns}
 \begin{aligned}
 (\underbrace{12\wedge 457}_{8})\vee (\underbrace{23\wedge 567}_{9}) \vee (\underbrace{34\wedge 617}_{10})\vee 7 &= 0 \\% 
 (\underbrace{12\wedge 456}_{11})\vee (23\wedge 576) \vee (34\wedge 716)\vee 6 &= 0 \\ %\\
 (12\wedge 465)\vee (23\wedge 675) \vee (\underbrace{34\wedge 715}_{12})\vee 5 &= 0 \\ %\\
 (12\wedge 564)\vee \underbrace{(23\wedge 674}_{13}) \vee (\underbrace{35\wedge 714}_{14})\vee 4 &= 0 \\ %\\
 (\underbrace{12\wedge 563}_{15})\vee (\underbrace{24\wedge 673}_{16}) \vee (\underbrace{45\wedge 713}_{17})\vee 3 &= 0 \\ %\\
 (\underbrace{13\wedge 562}_{18})\vee (\underbrace{34\wedge 672}_{19}) \vee (\underbrace{45\wedge 712}_{20})\vee 2 &= 0 \\ %\\
(\underbrace{23\wedge 561}_{21})\vee (34\wedge 671) \vee (45\wedge 721)\vee 1 &= 0, % & =\,  
\end{aligned}
\end{equation}
each of which states that four points are coplanar. Three of the four points are obtained by intersecting a line with a plane and the last point lies on the intersection of all three planes.

The set $\mathcal Q$ contains a total of 14 auxillary points, labelled $8$ through $21$ in Equation \eqref{eqn: 7eqns}.  The first equation reduces to 
\[
[1237][1457][2467][3567] - [1247][1357][2367][4567] = 0.
\]
If we fix a hyperplane $H$ not through point 7 and set $\overline{k}$ equal to the intersection of $H$ with the line joining point $k$ with point $7$, then the left hand side of this expression is a multiple of the polynomial from Theorem \ref{thm: pascal}, $[\overline{1}\,\overline{2}\,\overline{3}][\overline{1}\,\overline{4}\,\overline{5}][\overline{2}\,\overline{4}\,\overline{6}][\overline{3}\,\overline{5}\,\overline{6}] - [\overline{1}\,\overline{2}\,\overline{4}][\overline{1}\,\overline{3}\,\overline{5}][\overline{2}\,\overline{3}\,\overline{6}][\overline{4}\,\overline{5}\,\overline{6}]$, and so the equation encodes the condition that the six projected points $\overline{1}, \ldots, \overline{6}$ lie on a conic.

\end{ex}

% % %
\subsection{Matroid varieties via the Cayley-Bacharach theorems}
\label{sec: matroidVarsCB}

Pascal's Theorem is one of many manifestations of the Cayley-Bacharach Theorem.  The most well-known version of the Cayley-Bacharach Theorem is due to Chasles. It states that given the nine points of intersection of two cubics in $\PP^2$, if a third cubic passes through eight of the nine points then it necessarily passes through the ninth. Ren, Richter-Gebert, and Sturmfels \cite{rrgs} give methods to determine the ninth point given the 
other eight.  Eisenbud, Green, and Harris \cite{eisenbud+green+harris} give a historical survey of versions of the Cayley-Bacharach Theorem.  We use the following version from Traves \cite{traves}*{Theorem 6} to derive generalizations of Theorem \ref{thm: pascalQuartic}.

\begin{thm}[Cayley-Bacharach]
	\label{thm:cb}
	Suppose plane curves $\mathcal C,\mathcal D$ of degree $k$ meet in $k^2$ distinct points.  If $kd$ of those points are on an irreducible curve of degree $d$ then the remaining $k(k-d)$ points are on a curve of degree $k-d$.  Furthermore, if $\mathcal D$ factors into a product of linear forms then this curve is unique.
\end{thm}

\begin{ex}\label{ex: pascalCB}

Given three collinear points, choose another line through each of those points, forming a degenerate cubic $\mathcal C$. Then choose three more lines, one through each of the 3 points, forming a degenerate cubic $\mathcal D$. Applying Theorem \ref{thm:cb} with $k=3$ and $d=1$ to the three collinear points allows us to conclude that the six points of intersection in $\mathcal{C} \cap \mathcal{D}$ off the original line lie on a conic. This is just a restatement of the Braikenridge-Maclaurin Theorem, Theorem \ref{thm: pascal}. 
\end{ex}

A result from projective geometry need not be stated in terms of a Grassmann-Cayley expression to imply the existence of a nontrivial element of $I_x$ as we demonstrate in Theorem \ref{thm: pascalCB}.  To use Theorem \ref{thm:cb} in our construction we show that, just as there is a bracket condition that guarantees when six points lie on a conic, there are bracket conditions that characterize when points lie on higher degree curves. 

\begin{lem}
\label{lem: brackets}
	If $\binom{d+2}{2}$ points in $\PP^2$ lie on a curve of degree $d$, then the points satisfy a bracket polynomial of degree $\binom{d+2}{3}.$
\end{lem}

\begin{proof}
	Note that there are $\binom{d+2}{2}$ monomials of degree $d$ in three variables. Construct a square matrix $M$ of size $\binom{d+2}{2} \times \binom{d+2}{2}$ in which the $i^\text{th}$ row is obtained by evaluating each of these monomials at the $i^\text{th}$ point.
	
	Note that $\det M = 0$ precisely if there is a degree $d$ curve passing through the points.  Moreover, $\det M$ is invariant under change of basis so $\det M$ must be a bracket polynomial by the First Fundamental Theorem of Invariant Theory.  Because the degree of $\det M$ is $d\binom{d+2}{2} = 3\binom{d+2}{3}$ and each bracket contains three points, the degree of $\det M$ as a bracket polynomial is $\binom{d+2}{3}.$
\end{proof}

We extend Example \ref{ex: pascalCB} to $k >3$. First we carefully describe how to build a special arrangement of $k^2$ points. Starting with $k$ collinear points $P_1, \ldots, P_k$, we iteratively pick lines $\ell_{i}$ and $m_{i}$ through point $P_i$. Let $\mathcal{C}$ be the union of the lines $\ell_1,\ldots,\ell_k$ and $\mathcal{D}$ be the union of the lines $m_1,\ldots,m_k$. Since we are working over an infinite field, we can choose the lines $\ell_i$ and $m_i$ to avoid any given finite set of points. In particular, we can choose our lines so that the only collinearities among the $k^2$ points in $\mathcal{C} \cap \mathcal{D}$ are given by the $\ell_i$, the $m_i$, and the original line containing points $P_1,\dots,P_k$. Now $\mathcal{C} \cap \mathcal{D}$ is a set of $k^2$ points with $k$ original (collinear) points and $k^2-k$ \emph{residual points}.  Let $\M_{k^2}$ be the matroid associated to the $3 \times k^2$ matrix whose columns are homogeneous coordinates for the points in $\mathcal{C} \cap \mathcal{D}$, and $\V_{k^2}$ be the associated matroid variety. 

\begin{thm}
\label{thm: pascalCB}
There is a nontrivial polynomial of degree $\binom{k+1}{3}$ in the ideal $I_{k^2}$ of $\V_{k^2}$. 
\end{thm}

\begin{proof}

From Theorem \ref{thm:cb} the $k^2-k$ residual points lie on a curve of degree $k-1$. Pick $\binom{k+1}{2}$ of the residual points with the property that one of the lines $\ell_i$ in $\mathcal{C}$ contains exactly two of these points. From Lemma \ref{lem: brackets} there is a degree $\binom{k+1}{3}$ bracket polynomial $f$ that guarantees that the selected points lie on a degree $k-1$ curve. By construction $f \in I_{k^2}$. By Remark \ref{rmk: fnotinNx},   in order to show that $f$ is nontrivial, it is enough to find a configuration $y$ on which all the brackets in $N_{k^2}$ vanish and $f(y) \neq 0$. 

Let $y$ be the configuration obtained in the following way. Start with all $k^2$ points in $\mathcal{C} \cap \mathcal{D}$. Take the $k-2$ points on $\ell_i$ that are not among the chosen points on which $f$ vanishes and set those equal to the zero vector.  Now move the remaining two points on $\ell_i$ along the lines in $\mathcal D$, so that they no longer lie on $\ell_i$.  This resulting configuration guarantees the brackets in $N_{k^2}$ vanish.  However, working over an infinite field also guarantees we can move the two points so that our chosen set of $\binom{k+1}{2}$ points no longer lies on a curve of degree $k-1$.  
The total collection of $k^2$ points obtained this way will be our configuration $y$. Thus, $y$ is a configuration where all the brackets in $N_{k^2}$ vanish but $f(y) \neq 0$.  
\end{proof}

\begin{ex}
Figure \ref{fig:fourCapFour} illustrates Theorem \ref{thm: pascalCB} in the $k=4$ case. The four points on the dashed line are collinear, so the 12 residual points must lie on a cubic. By Lemma \ref{lem: brackets} we get up to $\binom{4}{2} \binom{3}{2} \binom{3}{2}=54$ nontrivial bracket polynomials of degree $10$. 
\end{ex}

\begin{figure}[h!t]
\begin{center}
\includegraphics[scale=0.2]{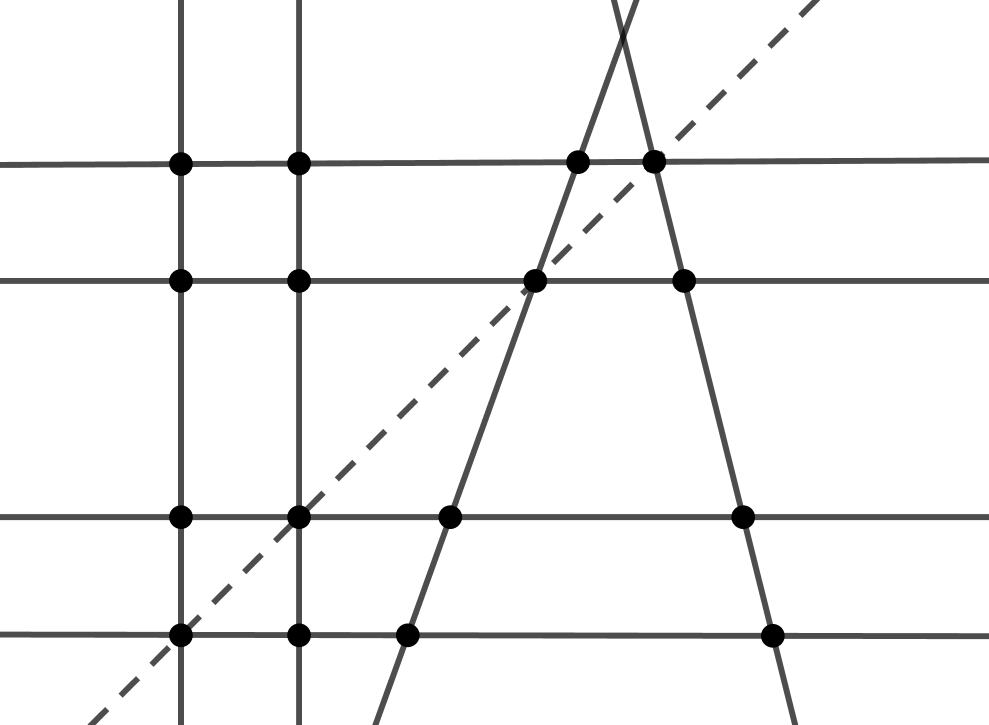}
\caption{An arrangement of points giving $\M_{4^2}$ built starting with 4 points lying on a (dashed) diagonal line. The union of the 4 horizontal lines is the curve $\mathcal{C}$ and the union of the remaining (solid) lines is the curve $\mathcal{D}$.}
\label{fig:fourCapFour}
\end{center}
\end{figure}

\noindent{\bf Acknowledgements.}
The authors would like to thank Paul Hacking for insightful conversations and Allen Knutson for his comments on a draft version of the paper. We are also grateful to the two anonymous referees, whose helpful comments  improved the paper. 


%\bibliography{matroidBib}
%\bibliographystyle{plain}
\end{document}